\documentclass[11pt]{article}

\usepackage{amsthm,amsmath,amssymb}
\usepackage{graphicx}
\usepackage[]{times}
\newcommand{\be}{\begin{equation}}
\newcommand{\ee}{\end{equation}} 
\newcommand{\bea}{\begin{eqnarray}}
\newcommand{\eea}{\end{eqnarray}} 
\theoremstyle{plain}
\newtheorem{theorem}{Theorem}
 
\newtheorem{corollary}{Corollary}

\theoremstyle{definition}

\theoremstyle{remark}

\def\1#1{^{(#1)}}

\def\be{\begin{equation}}
\def\ee{\end{equation}}
\def\bea{\begin{eqnarray}}
\def\eea{\end{eqnarray}}

\begin{document}
\title{\bf Gaussian Polynomials and \\Restricted Partition Functions with
Constraints}
\author{Leonid G. Fel\\
Department of Civil Engineering, Technion -- Israel Institute of Technology,\\ 
Haifa 32000, Israel\\
\small\tt lfel@technion.ac.il}
\date{}
\maketitle
\begin{abstract} 
We derive an explicit formula for a restricted partition function $P_n^m(s)$ 
with constraints making use of known expression for a restricted partition 
function $W_m(s)$ without constraints. \\ \\
{\bf Keywords:} restricted partition function, triangular Toeplitz matrix \\
{\bf 2010 Mathematics Subject Classification:} 05A17, 11P82.
\end{abstract}
Consider the linear Diophantine equation with constraints
\bea
a)\quad\sum_{r=1}^m r\; x_r=s,\qquad b)\quad\sum_{r=1}^m x_r\le n.\label{g1}
\eea  
Denote by $P_n^m(s)$ a number of the non-negative integer solutions $X=\{x_r\}$ 
of the linear system (\ref{g1}a,b). The following theorem dates back to 
Sylvester \cite{sy82} and Schur \cite{sh68},
\begin{theorem}\label{the1}
Let $P_n^m(s)$ be generated by the Gaussian polynomial $G(n,m;t)$ of the finite 
order $mn$
\bea
G(n,m;t)=\frac{\prod_{i=1}^{n+m}(1-t^i)}{\prod_{u=1}^{n}(1-t^u)\cdot\prod_{v=1}
^{m}(1-t^v)}=\sum_{s\ge 0}^{n m} P_n^m(s)\cdot t^s\;.\label{g2}
\eea
Then the partition function $P_n^m(s)$ has the following properties:
\bea
&&P_n^m(s)=0\quad\mbox{if}\quad s>n m,\qquad P_n^m(0)=P_n^m(n m)=1,\nonumber\\
&&P_n^m(s)=P_m^n(s)=P_n^m(n m-s),\qquad P_n^m\left(\frac{mn}{2}-s\right)=P_n^m
\left(\frac{mn}{2}+s\right),\label{g3}\\
&&P_n^m(s)-P_n^m(s-1)\ge 0\quad\mbox{for}\quad 0<s\le\frac{mn}{2}.\nonumber
\eea
\end{theorem}
A most comprehensive introduction to the Gaussian polynomials $G(n,m;t)$ and 
partition function $P_n^m(s)$ with constraints (\ref{g1}b) is given in 
\cite{and76}. But nowhere an explicit expression for $P_n^m(s)$ was derived. 
A situation is similar to the study of restricted partition function $W_m(s)$ 
without constraints (see its definition in (\ref{g4}) and following paragraphs) 
up to the last decade, when its explicit expression was finally found 
\cite{rf06}.

In this paper we derive such formula for $P_n^m(s)$ making use of strong 
relationship of the Gaussian polynomials with the Molien generating function for
restricted partition function $W_m(s)$ without constraints (\ref{g1}b).

Following \cite{fr02,rf06}, recall the basic facts from the partition theory and
consider the linear Diophantine equation $\sum_{r=1}^m d_r\;x_r=s$ without 
constraints (\ref{g1}b). Then the Molien function $M\left({\bf d}^m;t\right)$
reads,
\be
M\left({\bf d}^m;t\right)=\prod_{i=1}^m\frac{1}{1-t^{d_{i}}}=\sum_{s=0}^{\infty}
W\left(s,{\bf d}^m\right)\;t^s,\qquad {\bf d}^m=\{d_1,\ldots,d_m\}.\label{g4}
\ee
It generates a restricted partition function $W\left(s,{\bf d}^m\right)$ which 
gives a number of partitions of $s\ge 0$ into positive integers $\{d_1,\ldots,
d_m\}$, each not exceeding $s$, and vanishes, if such partition does not exist.
According to Proposition 4.4.1, \cite{st86} and Schur's theorem (see 
\cite{wi06}, Theorem 3.15.2), the function $W\left(s,{\bf d}^m\right)$ is a 
quasi-polynomial of degree $m-1$, 
\be
W\left(s,{\bf d}^m\right)=\sum_{r=0}^{m-1}K_r\left(s,{\bf d}^m\right)s^r,\qquad 
K_{m-1}\left(s,{\bf d}^m\right)=\frac{1}{(m-1)!\;\pi_m},\quad\pi_m=\prod_{i=1}
^md_i,\label{g5}
\ee
where coefficients $K_r\left(s,{\bf d}^m\right)$ are periodic functions with 
periods dividing $lcm(d_1,\ldots,d_m)$. The explicit expressions for $W\left(s,
{\bf d}^m\right)$ were derived in \cite{rf06} in a form of a finite sum over 
Bernoulli and Euler polynomials of higher order with periodic coefficients. 
Note that $W\left(0,{\bf d}^m\right)=1$.

In a special case, when ${\bf d}^m$ is a tuple of consecutive natural numbers 
$\{1,\ldots,m\}$, the expression for such partition function looks much more 
simple (see \cite{rf06}, formula (46)), and its straightforward calculations 
for $m=1,\ldots,12$ were presented in \cite{fr02}, section 6.1. For short, we 
denote it by $W_m(s)$. In particular, if $m$ is arbitrary large, we arrive at 
unrestricted partition function $W_{\infty}(s)$ known due to the Hardy-Ramanujan
asymptotic formula and Rademacher explicit expression \cite{and76}. By definition of 
a restricted partition function the following equality holds,
\bea
W_m(s)=W_{\infty}(s),\quad\mbox{if}\quad s\le m.\label{g6}
\eea
By comparison of two generating function in (\ref{g2}) and (\ref{g4}) we obtain,
\bea
G(n,m;t)M(m+n;t)=M(m;t)M(n;t),\quad M(m;t)=\prod_{i=1}^m\frac1{1-t^i}=\sum_{s=0}^{\infty}W_m(s)
\;t^s.\label{g7}
\eea
Substituting into (\ref{g7}) the polynomial representations (\ref{g2}) and 
(\ref{g4}) we arrive at
\bea
\sum_{s_1,s_2=0}^{\infty}P_n^m(s_1)W_{m+n}(s_2)\;t^{s_1+s_2}=\sum_{s_1,s_2=0}^
{\infty}W_n(s_1)W_m(s_2)\;t^{s_1+s_2}.\label{g8}
\eea
Equating in (\ref{g8}) the terms with different $s_1+s_2=g$ we obtain for every
$g$ a linear equation in $P_n^m(s)$,
\bea
\Delta=\sum_{s=0}^g\left[P_n^m(s)W_{m+n}(g-s)-W_n(s)W_m(g-s)\right]=0.\label{g9}
\eea
For small $s$ equation (\ref{g9}) may be resolved easily.
\begin{theorem}\label{the2}
Let $P_n^m(g)$ be generated by the Gaussian polynomial $G(n,m;t)$, then
\bea
P_n^m(g)&=&W_{\mu_1}(g),\qquad 0\le g\le\mu_1,\qquad\mu_1=\min(n,m),\label{g10}
\\
P_n^m(g)&=&W_{\mu_1}(g),\qquad\mu_1\le g\le\mu_2,\qquad\mu_2=\max(n,m),
\label{g11}\\
P_n^m(g)&=&W_n(g)+W_m(g)-W_{n+m}(g),\quad\mu_2\le g\le n+m.\label{g12}
\eea
\end{theorem}
\begin{proof}
Start with observation in (\ref{g2},\ref{g3}) that due to the invariance of 
$G(n,m;t)$ and $P_n^m(g)$ under interchange $n\leftrightarrow m$, it is enough 
to prove (\ref{g10},\ref{g11},\ref{g12}) in the case $n\le m$. 

If $\;\;0\le s\le g\le n$, then due to (\ref{g6}) we have $W_m(g-s)=W_{m+n}(g-
s)=W_{\infty}(g-s)$ since $0\le g-s\le n$. Substitute the above equalities into 
(\ref{g9}) and obtain $P_n^m(g)=W_n(g)$.

If $\;\;n\le g\le m$, then due to (\ref{g6}) we have $W_m(s)=W_{m+n}(s)=W_{
\infty}(s)$ for all $0\le s\le g$. Substitute the last equalities into 
(\ref{g9}) and obtain $P_n^m(g)=W_n(g)$.

If $\;\;m\le g\le m+n\;$ and $\;g=m+r$, $\;1\le r\le n$, let us represent the 
l.h.s. of (\ref{g9}) as follows, $\Delta=\Delta_1+\Delta_2$,
\bea
\Delta_1=\left(\sum_{s=0}^r+\sum_{s=m}^{m+r}\right)\left[P_n^m(s)W_{m+n}
(m+r-s)-W_n(s)W_m(m+r-s)\right]=\sum_{s=0}^rF_1(s),\quad\label{g13}\\
F_1(s)=P_n^m(s)W_{m+n}(m+r-s)-W_n(s)W_m(m+r-s)+P_n^m(m+r-s)W_{m+n}(s)-
\nonumber\\
W_n(m+r-s)W_m(s),\hspace{1cm}\nonumber\\
\Delta_2=\sum_{s=r+1}^{m-1}\left[P_n^m(s)W_{m+n}(m+r-s)-W_n(s)W_m(m+r-s)\right],
\qquad\mbox{i.e.,}\hspace{2.2cm}\nonumber\\
\Delta_2=\sum_{s=r+1}^kF_1(s),\quad\mbox{if}\quad g=2k+1;\qquad
\Delta_2=\sum_{s=r+1}^{k-1}F_1(s)+F_2(k),\quad\mbox{if}\quad g=2k,
\hspace{.6cm}\label{g14}\\
F_2(s)=P_n^m(s)W_{m+n}(s)-W_n(s)W_m(s).\hspace{7.7cm}\label{g15}
\eea
Prove that the term $\Delta_2$ vanishes and start with $F_2(k)$. Since $n\le m$
then $(m+1)/2\le k\le m$, that implies two equalities: first, by (\ref{g6}), $W_
{m+n}(k)=W_m(k)=W_{\infty}(k)$ and next, by (\ref{g11}), $P_n^m(k)=W_n(k)$ 
either $k\le n$ or $n\le k\le m$. Substitute these equalities into (\ref{g15}) 
and obtain $F_2(k)=0$.

Consider the term $\Delta_2$ with $g=2k$ in (\ref{g14}), where $1\le r<s<k\le 
m$ and $n\le m+r-s\le m$, and write four equalities:
\bea
&&W_{m+n}(m+r-s)=W_m(m+r-s),\qquad W_{m+n}(s)=W_m(s),\nonumber\\
&&P_n^m(m+r-s)=W_n(m+r-s),\qquad P_n^m(s)=W_n(s).\label{g16}
\eea
which implies $\Delta_2=0$. The term $\Delta_2$ with $g=2k+1$ in (\ref{g14})
vanishes also by the same reasons (\ref{g16}).

Thus, instead of (\ref{g9}), we arrive at equation $\sum_{s=0}^rF_1(s)=0$ where
$0\le s\le r\le n$. However, according to (\ref{g6}) we have $W_{m+n}(s)=W_m(s)
=W_n(s)=W_{\infty}(s)$, so we obtain
\bea
\sum_{s=0}^rW_{\infty}(s)\left[W_{m+n}(m+r-s)-W_m(m+r-s)+P_n^m(m+r-s)-W_n(m+r-s)
\right]=0,\nonumber
\eea
where $m<m+r-s\le m+n$. Denote $g=m+r-s$ and arrive at (\ref{g12}) that proves 
Theorem.
\end{proof}
\begin{corollary}\label{cor1}
\bea
\lim_{m\to\infty}P_n^m(s)=W_n(s),\qquad
\lim_{n\to\infty}P_n^m(s)=W_m(s).\nonumber
\eea
\end{corollary}
Further extension of Theorem \ref{the2} on higher $s\ge m+n$ loses its 
generality, that indicates a necessity to develop another approach. Equations 
(\ref{g9}) with different $g$ represent the linear convolution equations with a 
triangular Toeplitz matrix. They can be solved using the inversion of the 
Toeplitz matrix (see \cite{goh74}, Chapt. 3). We will give another 
representation of a general solution of (\ref{g9}) which can be found due to 
triangularity of the Toeplitz matrix.
\begin{theorem}\label{the3}
Let $P_n^m(s)$ be generated by the Gaussian polynomial $G(n,m;t)$, then 
\bea
P_n^m(g)=\sum_{r=0}^{g-1}\left[\sum_{s=0}^{g-r}W_n(s)W_m(g-r-s)-W_{m+n}(g-r)
\right]\Phi_r(m+n),\label{g17}
\eea
where coefficients $\Phi_r(m+n)$ are related to $W_{m+n}(s)$ and defined as 
follows,
\bea
\Phi_r(m+n)=\sum_{q=1}^r\sum_{q_1,\ldots,q_r}^q\frac{q!}{q_1!\cdot\ldots\cdot 
q_r!}\prod_{k=1}^r(-1)^{q_k}W^{q_k}_{m+n}(k),\quad\sum_{k=1}^rkq_k=r,\quad
\sum_{k=1}^rq_k=q.\quad\label{g18}
\eea
\end{theorem}
\begin{proof}
Consider linear convolution equations with a triangular Toeplitz matrix
\bea
P(g)=T(g)+\sum_{s=0}^{g-1}P(s)\;U(g-s),\qquad P(0)=T(0)=1,\label{g19}
\eea
where two known functions $T(g),\;U(g)$ and unknown function $P(g)$ are
considered only on the non-negative integers. The successive recursion of
(\ref{g19}) gives
\bea
P(g)&=&T(g)+T(g-1)\;U(1)+\sum_{s=0}^{g-2}P(s)\cdot[U(g-s)+U(1)\;U(g-1-s)]
\nonumber\\
&=&T(g)+T(g-1)\;U(1)+T(g-2)\;[U(2)+U^2(1)]+\nonumber\\
&&\sum_{s=0}^{g-3}P(s)\;[U(g-s)+U(1)\;U(g-1-s)+[U(2)+U^2(1)]\; U(g-2-s)].
\nonumber
\eea
By induction we can arrive at
\bea
P(g)=\sum_{r=0}^{k-1}T(g-r)\;\Phi_r(U)+\sum_{s=0}^{g-k}P(s)\;\sum_{r=0}^{k-1}
U(g-r-s)\;\Phi_r(U),\qquad 1\le k\le g,\label{g20}
\eea
where polynomials $\Phi_r(U)$ are related to the restricted partition number
$W_g(r)$ of positive integer $r$ into non-negative parts, none of which exceeds 
$g$, 
\bea
\Phi_r(U)=\sum_{q=1}^r\sum_{q_1,\ldots,q_r}^q\frac{q!}{q_1!\;\ldots\;q_r!}
\prod_{k=1}^r U^{q_k}(k),\qquad\sum_{k=1}^r kq_k=r,\quad\sum_{k=1}^rq_k=q,
\label{g21}
\eea
A sum $\sum_{q_1,\ldots,q_r}^q$ in (\ref{g21}) is taken over all distinct
solutions $\{q_1,\ldots,q_r\}$ of the two Diophantine equations (\ref{g21}) with
fixed $q$. Below we present expressions for the four first polynomials $\Phi_r
(U)$,
\bea
&&\Phi_0(U)=1,\quad\Phi_1(U)=U(1),\quad\Phi_2(U)=U(2)+U^2(1),\nonumber\\
&&\Phi_3(U)=U(3)+2U(2)U(1)+U^3(1),\label{g22}
\eea
and in (\ref{g27}) for the other two. The total number of algebraically 
independent terms, contributing to the polynomial $\Phi_r(U)$, is equal $W_r(r)
$, while the sum of coefficients at these terms is equal $2^{r-1}$, e.g., $W_3(
3)=3$, $1+2+1=2^2$. The terms comprising $\Phi_r(U)$ may be calculated with 
Mathematica Software using {\sf IntegerPartitions[r]}.

Put $k=g$ into (\ref{g20}) and obtain finally,
\bea
P(g)=\sum_{r=0}^{g-1}[T(g-r)+U(g-r)]\cdot\Phi_r(U).\label{g23}
\eea
Comparing (\ref{g19}) and (\ref{g8}) we conclude,
\bea
T(g)=\sum_{s=0}^gW_n(s)\;W_m(g-s),\qquad U(s)=-W_{m+n}(s).\label{g24}
\eea
Substituting (\ref{g24}) into (\ref{g23}) we immediately arrive at (\ref{g11},
\ref{g12}).
\end{proof}
Illustrate the usage of formulas (\ref{g17},\ref{g18}) and apply them to 
calculate the partition function $P_n^m(s)$ with small $m,n$, e.g., 
$m\!=2,n\!=3$,
\bea
P_3^2(g)&=&\sum_{r=0}^{g-1}\left[\sum_{s=0}^{g-r}W_3(s)W_2(g-r-s)-W_5(g-r)
\right]\Phi_r(5).\label{g25}
\eea
For this aim we need expressions for $W_2(s)$, $W_3(s)$ and $W_5(s)$, found 
in \cite{fr02}, section 6.1,
\bea
W_2(s)&=&\frac{s}{2}+\frac{3}{4}+\frac1{4}\cos\pi s,\nonumber\\
W_3(s)&=&\frac{s^2}{12}+\frac{s}{2}+\frac{47}{72}+
\frac{2}{9}\cos\frac{2\pi s}{3}+\frac1{8}\cos\pi s,\nonumber\\
W_5(s)&=&\frac{s^4}{2880}+\frac{s^3}{96}+\frac{31s^2}{288}+\frac{85s}{192}+
\frac{s}{64}\cos\pi s+\frac{50651}{86400}+\frac1{16}\left(\cos\frac{\pi s}{2}+
\sin\frac{\pi s}{2}\right)+\nonumber\\
&&\frac{2}{27}\cos\frac{2\pi s}{3}+\frac{2}{25}\cos\frac{4\pi s}{5}+
\frac{15}{128}\cos\pi s,\label{g26}
\eea
and polynomials $\Phi_r(U)$, $0\le r\le 5$, presented in (\ref{g22}) and 
also given below,
\bea
\Phi_4(U)&=&U(4)+2U(3)U(1)+U^2(2)+3U(2)U^2(1)+U^4(1),\label{g27}\\
\Phi_5(U)&=&U(5)+2U(4)U(1)+2U(3)U(2)+3U(3)U^2(1)+3U^2(2)U(1)+\nonumber\\
&&4U(2)U^3(1)+U^5(1).\nonumber
\eea
Making use of (\ref{g26}), calculate $W_2(r)$, $W_3(r)$, $W_5(r)$, $0\le r\le 
5$, and find 
$$
\Phi_0(5)=\Phi_5(5)=1,\quad\Phi_1(5)=\Phi_2(5)=-1,\quad\Phi_3(5)=\Phi_4(5)=0.
$$
Substitute the values for $W_k(r)$, $k=2,3,5$, and $\Phi_r(5)$, $0\le r\le 5$, 
into (\ref{g25}) and obtain
\bea
P_3^2(0)=P_3^2(1)=P_3^2(5)=P_3^2(6)=1,\qquad P_3^2(2)=P_3^2(3)=P_3^2(4)=2,
\label{g28}
\eea
that satisfies the straightforward calculation of the Gaussian polynomial
$G(3,2;t)$,
\bea
G(3,2;t)=1+t+2t^2+2t^3+2t^4+t^5+t^6.\nonumber
\eea
It is easy to verify that the values in (\ref{g28}) satisfy also Theorem 
\ref{the2}.
\section*{Acknowledgement}
The useful discussions with the late Prof. I. Gohberg are highly appreciated. 
The research was supported by the Kamea Fellowship.

\end{document}